\documentclass[11pt]{article}
\usepackage[english]{babel}
\usepackage[utf8]{inputenc}
\usepackage{amsmath, amsfonts, amsthm, amssymb, amscd,enumerate}
\usepackage{graphicx, color}
\newtheorem{theorem}{Theorem}
\newtheorem{prop}{Proposition}
\newtheorem{lemma}{Lemma}
\newtheorem{Cor}{Corollary}
\newtheorem{nb}{Remark}

\newtheorem{Def}{Definition}
\DeclareMathOperator*{\dis}{dist}
\DeclareMathOperator*{\real}{Re}
\DeclareMathOperator*{\Imm}{Im}

\begin{document}
\title{On Brolin's theorem over the quaternions}
\date{}
\author{C. Bisi, A.De Martino}
\maketitle
\begin{abstract}
\noindent
In this paper we investigate the Brolin's theorem over $\mathbb{H}$, the skew field of quaternions. Moreover, considering a quaternionic polynomial $p$ with real coefficients, we focus on the properties of its equilibrium measure, among the others, the mixing property and the Lyapunov exponents of the measure. We prove a central limit theorem and we compute the topological entropy and measurable entropy with respect to the quaternionic equilibrium measure. We prove that they are equal considering both a  quaternionic polynomial with real coefficients and a polynomial with coefficients in a slice but not all real. Brolin's theorems for the one slice preserving polynomials and for generic polynomials are also proved.
\end{abstract}

\section{Introduction}
In 1965 Hans Brolin proved a very important theorem in complex dynamical systems \cite{B}. This remarkable result deals with the distribution of preimages of generic points for holomorphic functions.

This theorem was studied in several contexts, for instance Lyubich \cite{L}, Freire, Lopez and Mané \cite{FLM} generalized Brolin's theorem to rational maps in $ \mathbb{P}^{1}_{\mathbb{C}}$ and Favre and Jonsson proved a version of Brolin's theorem in two complex dimension \cite{FJ}. An overview over the equilibrium measure can be found in \cite{B2}.

In the case of a complex polynomial $p(z)$ of degree $d \geq 2$, Brolin's theorem asserts the existence of a unique equilibrium measure of the polynomial with the property that the preimages of the iterates of the polynomial $p(z)$ of a generic point of the plane are equally distributed. In the complex case this kind of measure is linked to several dynamical properties of the polynomial: its support is related to the Julia set of the polynomial, it is an invariant measure, it is the unique measure of maximal entropy equal to $ \log d$ and it satisfies the central limit theorem \cite{DS, DS1, DS2}.

The aim of this paper is to generalize these results in the quaternionic setting. Now, in order to state the Brolin's theorem we give two preliminary definitions over $\mathbb{C}.$
\newline
\newline
\textbf{Definition}
$$ G_{p}(z)= \lim_{n \to + \infty} G_{n}(z), \qquad  \forall z \in \mathbb{C},$$
where
$$ G_{n}(z)= d^{-n} \log^{+}|p^{n}(z)|,$$
and $p^{n}$ is the n-times composition of $p$.
\newline
\newline
\textbf{Definition}
Given a function $f$ from a set $X$ to itself, the set $ \mathcal{E}$ is called exceptional if it is finite and if
$$ f(\mathcal{E})=\mathcal{E} \qquad f^{-1}(\mathcal{E})=\mathcal{E}.$$
\newline
\newline
If $f$ is a polynomial and $X=\mathbb{C}$ the exceptional set $\mathcal{E}$ is uniquely defined and its cardinality is less or equal than $2$.
Now, we are ready to state the complex Brolin's theorem.
\newline
\newline
\textbf{Theorem}[Brolin's theorem]
\label{start}
Let $p$ a polynomial $p : \mathbb{C} \to \mathbb{C}$ with degree $d \geq 2$. Then for all $ a \in \mathbb{C} \setminus \mathcal{E} $ the following measure
$$ \mu_{n}:= d^{-n} \sum_{b \in p^{-n}(a)} \delta_{b},$$
converges weakly to the so-called equilibrium measure $ \mu$; i.e.
$$ \mu_{n} \rightharpoonup \mu:= \Delta G_{p}.$$
The equilibrium measure does not depend on the choice of $a \in \mathbb{C} \setminus \mathcal{E} $.
\newline
\newline

The plan of the paper is the following: in section 2 we recall some basic notions over quaternionic functions. In section 3 we deal with the proofs of some properties of the operator $ \Delta_{*}$ (introduced by C.Bisi and J.Winkelmann in \cite{BW}) which are going to be useful in section 4 where is proved the Brolin's theorem in the slice preserving case.

Let us consider the quaternionic polynomial $p$ with real coefficients, we denote by $p_I$ the restriction of the polynomial to the complex plane $ \mathbb{C}_{I}$, defined by
$$  \mathbb{C}_I:= \{ q \in \mathbb{H} \,\, | \,\, q= \alpha+I \beta, \, \, \, \alpha, \beta \in \mathbb{R} \}.$$
We will prove in section 2 that $ \mathbb{C}_I$ is $p_I$ -invariant for all $I \in \mathbb{S}$, if $p$ has real coefficients.
We define by $ \mathcal{E}_I$ the exceptional set of the polynomial $p_I$.
\newline
\newline
\textbf{Theorem}
Let $p$ be a polynomial with real coefficients with degree $d \geq 2$. Then for all $I \in \mathbb{S}= \{q \in \mathbb{H}, \, q^2=-1 \}$ and for every $ a \in \mathbb{\mathbb{R}} \setminus (\mathcal{E}_I \cap \mathbb{R})$ the measure
$$ \mu_{n}:= d^{-n} \sum_{b \in p^{-n}(a)} \delta_{b},$$
converges weakly to the equilibrium measure $ \mu$, where $ \delta_b$ is the Dirac delta centred in $b$ if $b$ is a point, otherwise it is the Lebesgue measure of the sphere if $b$ is a sphere. This means that
$$ \mu_{n} \rightharpoonup \mu:= \Delta_{*} G_{p}.$$
The equilibrium measure does not depend on the choice of $a \in \mathbb{\mathbb{R}} \setminus (\mathcal{E}_I \cap \mathbb{R})$.
\newline
\newline
We remark that it is possible to write the equilibrium measure as an integral form: $ \mu= \frac{1}{4 \pi} \int_{I \in \mathbb{S}} \mu_{I} \, dI$.
\\In section 5 we show the Brolin's theorem in the one slice preserving  case, the idea is to generalize to all the slices what happens on a slice. We define $ \mathcal{E}'_I$ the exceptional set of $g_n:=(\mathcal{P}^{n}_{I})^{s}$ restricted to the slice $ \mathbb{C}_I$.
\newline
\newline
\textbf{Theorem}
Let $\mathcal{P}_I$ be a polynomial with degree $d \ge 2$ and all coefficients in $\mathbb{C}_I$ but not all real. Then we have for all $a \in \mathbb{R} \setminus ( \mathcal{E}'_I \cap \mathbb{R}):$
$$  \frac{1}{2}\Delta_* \log^{+} |g_{n}-a| \rightharpoonup \frac{1}{8 \pi} \int_{I \in \mathbb{S}} \mu_{\mathfrak{P}(I)} \, dI + \frac{1}{8 \pi} \int_{I \in \mathbb{S}} \mu_{\mathfrak{P}^{c}(I)} \, dI:= \mu',$$
where if $\mathcal{P}_I(q)=q^d a_d(I)+...+a_0(I)$ with $ q \in \mathbb{H}$ and $a_{\ell}(I)=x_\ell+I y_\ell$ ($x_\ell, y_\ell \in \mathbb{R}$) for $ \ell =0,...,d$, then, for all $J \in \mathbb{S}$, we have $ \mathfrak{P}(q, J)=q^da_d(J)+...+a_0(J)$ with $ q \in \mathbb{C}_J$ and $a_{\ell}(J)=x_\ell+J y_\ell$ ($x_\ell, y_\ell \in \mathbb{R}$) for $ \ell =0,...,d$.
We pointed out that $I$ in $\mathcal{P}_I(q)$ is fixed while in $\mathfrak{P}(q, J)$ the imaginary unit $J$ is a variable. We call the limit $ \mu'$ \emph{the Equilibrium measure of $\mathcal{P}_I$} in this case and it does not depend on the chosen $a \in \mathbb{R}\setminus (\mathcal{E}'_I \cap \mathbb{R}).$
\newline
\newline
In section 6 we consider the following composition \cite[Def. 2.1]{GSS1}, \cite{CSS1}.
\newline
\newline
\textbf{Definition}
Denoting $g(q)= \sum_{n=0}^{\infty} q^n a_n$ and $ w(q)=\sum_{n=0}^{\infty} q^n b_n$. We define
$$ (g \bullet w)(q)= \sum_{n=0}^\infty (w(q))^{*n} a_n,$$
where $(w(q))^{*n}$ means that we take the $n$-th power with respect to the $*$- product, which is the operation in the quaternionic setting that gives us a regular function from the multiplication of two regular functions (see \eqref{ref3}). We recall that a function is slice regular if it is holomorphic in every slice, i.e. in the complex planes $ \mathbb{C}_I$ when the imaginary unit varies.
This helps us to give a sort of Brolin's theorem for a quaternionic polynomial with coefficients in different slices. For this kind of polynomials we defined the exceptional set as
$$ \mathcal{E''}:= \{q:\, |[q]|< \infty \}.$$

where $[q]$ is the orbit, defined as
$$ [q]:= \{h_n(q)| \, \, n \in \mathbb{N}\}$$
and $h_n(q):=(p^{\bullet n})^{s}$.
\newline
\newline
\textbf{Theorem}
Let $p(q)$ be a slice regular polynomial with degree $d \ge 2$ and coefficients in different slices. Then, for all $a \in \mathbb{R} \setminus (\mathcal{E''} \cap \mathbb{R})$ and $b \in \mathbb{R} \setminus (\mathcal{E''} \cap \mathbb{R})$ with $a \neq b $  we have that
$$ d^{-n} \log|h_n(q)-a|-d^{-n} \log|h_n(q)-b| \to 0.$$

However in this general case we cannot give a limit measure.
\\ We can provide a counterexample where the equidistribution property fails. Let us consider $p(q)=q^n$ and $a \in \mathbb{H} \setminus \mathbb{R}$. If $a \in \mathbb{C}_I$, in the slice $ \mathbb{C}_I$ we obtain the equilibrium measure $ \mu_I$. On the other hand if $ a \in \mathbb{C}_J$, in the slice $ \mathbb{C}_J$ we obtain the equilibrium measure $ \mu_J$.
This means that the equilibrium measure depends on the choice of the point.
\\In section 7 the main result is the proof of the fact that the equilibrium measure is mixing, this means that if $\mu$ is the equilibrium measure and $ \varphi, \psi \in C^0(\mathbb{H})$ we have
$$ \langle \mu, ( \psi \circ p^{n}) \cdot \varphi \rangle \to \langle \mu, \varphi \rangle \langle \mu, \psi \rangle.$$
This happens both in the slice preserving and in the one slice preserving case. In section 8 we prove a central limit theorem for a quaternionic polynomial with real coefficients and for the same kind of polynomials in the last section we evaluate their topological entropy and measurable entropy with respect to the quaternionic equilibrium measure. We observed that in this case the two entropies are equal to $ \log d$ ($d$ is the degree of the polynomial). In subsection 9.1 is proved that also for polynomials $\mathcal{P}_I$ with coefficients in a slice but not all real the topological entropy and measurable entropy are equal to $ \log d$.

\section{Prerequisites about quaternionic functions}

In this section we will overview and collect the main notions and results needed for our aims.
First of all, let us denote by $ \mathbb{H}$ the real algebra of the quaternions. An element $q \in \mathbb{H}$ is usually written as $q=x_0+ix_1+jx_2 +kx_3$, where $i^2=j^2=k^2=-1$ and $ijk=-1$. Given a quaternion $q$ we introduce a conjugation in $ \mathbb{H}$ (the usual one), as $q^c=x_0-ix_1-jx_2 -kx_3$; with this conjugation we define the real part of $q$ as $Re(q):= (q+q^c)/2$ and the imaginary part as $Im(q)=(q-q^c)/2$. With the defined conjugation we can write the euclidian square norm of a quaternion $q$ as $|q|^2=qq^c$. The subalgebra of real numbers will be identified, of course, with the set $ \mathbb{R}= \{ q \in \mathbb{H}| \, Im(q)=0 \}$

Now, if $q$ is such that $Re(q)=0$, then the imaginary part of $q$ is such that $ (Im(q)/|Im(q)|)^2=-1$. More precisely, any imaginary quaternion $I=ix_1+jx_2 +kx_3$, such that $x^2_1+x^2_2+x^2_3=1$ is an imaginary unit. The set of imaginary units is then a real 2-sphere and it will be conveniently denoted as follows
$$ \mathbb{S}:= \{ q \in \mathbb{H}\,\,| \,\, q^2=-1 \}= \{q  \in \mathbb{H}\,\,| \,\, Re(q)=0, \, |q|=1 \}.$$
With the previous notation, any $q \in \mathbb{H}$ can be written as $q= \alpha+I \beta$, where $ \alpha, \beta \in \mathbb{R}$ and $ I \in \mathbb{S}$.

Given any $I \in \mathbb{S}$ we will denote the real subspace of $ \mathbb{H}$ generated by 1 and $I$ as
$$ \mathbb{C}_I:= \{ q \in \mathbb{H} \,\, | \,\, q= \alpha+I \beta, \, \, \, \alpha, \beta \in \mathbb{R} \}.$$
Sets of the previous kind will be called \emph{slices} and they are also complex planes. All these notations reveal now clearly the \emph{slice} structure of $ \mathbb{H}$ as union of complex planes $ \mathbb{C}_I$ for $I$ which varies in $ \mathbb{S}$, i.e.
$$ \mathbb{H}= \bigcup_{I \in \mathbb{S}}\mathbb{C}_{I}, \quad \bigcap_{I \in \mathbb{S}} \mathbb{C}_I= \mathbb{R}.$$

We denote the real 2-sphere with center $ \alpha \in \mathbb{R}$ and radius $| \beta|$ with $ \beta \in \mathbb{R}$ (passing through $ \alpha+I\beta \in \mathbb{H}$), as:
\begin{equation}
\label{rr1}
\mathbb{S}_{\alpha+I \beta}:= \{ q \in \mathbb{H} \,\, |\,\,  q= \alpha+I \beta, \,\,\, I \in \mathbb{S} \} \subset \mathbb{H}.
\end{equation}
Obviously, if $\beta=0$, then $ \mathbb{S}_{\alpha}= \{ \alpha \}$.

Now, we are going to introduce the main definitions and features of slice functions, following \cite{GP}.
\\ The complexification of $ \mathbb{H}$ is defined to be the real tensor product between $\mathbb{H}$ itself and $ \mathbb{C}$:
$$ \mathbb{H}_\mathbb{C}:= \mathbb{H} \otimes_{\mathbb{R}} \mathbb{C}:= \{p+\iota q| \quad p,q \in \mathbb{H} \}.$$
In $ \mathbb{H}_\mathbb{C}$, the following associative product is defined: if $p_1+\iota q_1$, $p_2+\iota q_2$ belong to $ \mathbb{H}_{\mathbb{C}}$ then,
$$ (p_1+\iota q_1)(p_2+\iota q_2)=p_1p_2-q_1q_2+\iota (p_1q_2+q_1p_2).$$
The usual complex conjugation $ \overline{p+ \iota q}=p- \iota q$ commutes with the following involution $(p+\iota q)^c= p^c+ \iota q^c$, where $p^c$ and $q^c$ are quaternionic conjugates.

We introduce now a class of subsets of $ \mathbb{H}$ also investigated in \cite{BW2}

\begin{Def}
Given any set $D \subseteq \mathbb{C}$, we define its circularization (or axially symmetric completion) as the subset in $ \mathbb{H}$ defined as follows
$$ \Omega_D:= \{ \alpha+I \beta| \, \, \alpha+i \beta \in D, I \in \mathbb{S} \}.$$
Such subsets of $ \mathbb{H}$ are called circular sets  or (axially symmetric). If $D \subseteq \mathbb{C}$ is such that $D \cap \mathbb{R} \neq \emptyset$, then $\Omega_D$ is also called slice domain (see \cite{GP}).
\end{Def}
From now on, $ \Omega_D \subseteq \mathbb{H}$ will always denote a circular domain. We can state now the following definition \cite{GP,S1}.
\begin{Def}
\label{Stem}
Let $ D \subseteq \mathbb{C}$ be any symmetric set with respect to the real line. A function $F=F_1+ \iota F_2: D \to \mathbb{H}_{\mathbb{C}}$ such that $F(\bar{z}) = \overline{F(z)}$ is said to be a stem function.
\\A function $f: \Omega_D \to \mathbb{H}$ is said to be a (left) slice function if it is induced by a stem function $F=F_1+ \iota F_2$ defined on $D$ in the following way: for any $ \alpha+I \beta \in \Omega_D$
$$ f( \alpha+I \beta)=F_1  (\alpha+ i \beta)+IF_2( \alpha+ i \beta).$$
If a stem function $F$ induces the slice function $f$, we will write $f= \mathcal{I}(F)$. The set of slice functions defined on a certain circular domain $\Omega_D$ will be denoted by $ \mathcal{S}( \Omega_D)$. Moreover, we denote by $ \mathcal{S}^k( \Omega_D)$ the set of slice function of class $C^k$, with $k \in \mathbb{N} \cup \{ \infty \}$.
\end{Def}

Let now $D \subseteq \mathbb{C}$  be an open set and $z= \alpha+i \beta \in D$. Given a stem function $F=F_1+ \iota F_2:D \to \mathbb{H}_{\mathbb{C}}$ of class $C^1$ then
$$ \frac{\partial F}{\partial z}, \frac{\partial F}{\partial \bar{z}}:D \to \mathbb{H}_{\mathbb{C}},$$
defined as
$$ \frac{\partial F}{\partial z}= \frac{1}{2} \left( \frac{\partial F}{\partial \alpha}- \iota \frac{\partial F}{\partial \beta} \right) \quad \hbox{and} \quad \frac{\partial F}{\partial \bar{z}}= \frac{1}{2} \left( \frac{\partial F}{\partial \alpha}+ \iota \frac{\partial F}{\partial \beta} \right)$$
are stem functions. The previous stem functions induce the continuous \emph{slice derivatives}
$$ \partial_c f= \mathcal{I} \left( \frac{\partial F}{\partial z} \right), \quad \overline{\partial_c} f= \mathcal{I} \left( \frac{\partial F}{\partial \bar{z}} \right).$$

\begin{Def}
Let $\Omega_D$ be a circular open set. A function $f= \mathcal{I}(F) \in \mathcal{S}^1(\Omega_D)$ is (left) slice regular if its stem function $F$ is holomorphic, i.e. each of the 4 complex components of $F$ is holomorphic. The set of slice regular functions will be denoted by
$$ \mathcal{SR}(\Omega_D):= \{f \in \mathcal{S}^1(\Omega_D)| f= \mathcal{I}(F), F:D \to \mathbb{H}_{\mathbb{C}}\quad \hbox{holomorphic} \}.$$

Equivalently, a slice function $f \in \mathcal{S}^1 (\Omega_D)$ is slice regular if the following equation holds \cite{GSS}
$$ \overline{\partial_c}f( \alpha+ J \beta)=0, \quad \forall \alpha+J \beta \in \Omega_D.$$
\end{Def}
In order to be precise the first time the slice regular theory appeared in a book was in 2011, see \cite{CSS}.

It is now well-known that the $*$-product between two power series in the variable $q \in \mathbb{H}$ coincides with their convolution product i.e.  if $f(q)= \sum_{j} q^j a_j$ and $g(q)= \sum_{k}q^k b_k$ are converging power series with coefficients $a_j, b_k \in \mathbb{H}$, then
\begin{equation}
\label{ref3}
(f*g)(q):= \sum_{n} q^n \left( \sum_{j+k=n}a_jb_k \right).
\end{equation}
\begin{Def}
\label{ref2}
A slice function $f= \mathcal{I}(F)$ is called \emph{slice-preserving} if, for all $J \in \mathbb{S}$, $f(\Omega_D \cap \mathbb{C}_J) \subseteq \mathbb{C}_J$.
\end{Def}
It is now easy to see that if $f$ is a slice preserving and $g$ is any slice function then $fg=f*g=g*f$. If both $f$ and $g$ are slice preserving, then $fg=f*g=g*f=gf$.
This result will be very useful.
\begin{lemma}
\label{refee1}
A polynomial is slice preserving if and only if it has real coefficients.
\end{lemma}
\begin{proof}
If $p$ has real coefficients by definition it is trivial that it is slice preserving. The viceversa, since the polynomial is analytic it is possible to write its series expansion, which has real coefficients by hypothesis. From the identity principle of complex analysis it is possible to write the series with a complex variable. Thus by definition of slice preserving function we have the thesis.
\end{proof}
\begin{Def}
\label{ref5}
A slice regular function $f$ is one-slice preserving if there exists $J \in \mathbb{S}$ such that $f(\Omega_D \cap \mathbb{C}_J) \subset \mathbb{C}_J$; for a fixed $J \in \mathbb{S}$, these functions will also be called $\mathbb{C}_J$-preserving.
\end{Def}
Examples of such kind of functions are quaternionic polynomials and quaternionic power series with right coefficients which belong to $ \mathbb{C}_J$ .

Given any quaternionic function $f: \Omega_D \subseteq \mathbb{H} \to \mathbb{H}$ of one quaternionic variable we will denote its zero set in the following way:
\begin{equation}
\label{refee}
\mathcal{Z}(f):= \{ q \in \Omega_D \, \,|\, \, \,f(q)=0 \}.
\end{equation}
It is possible to express the $*$-product of two slice functions in terms of their punctual product properly evaluated. The next proposition clarifies this fact, its proof can be seen in the book \cite{GSS}.
\begin{prop}
Let $f,g \in \mathcal{SR}(\Omega_D)$ then, for any $q \in \Omega_D \setminus \mathcal{Z}(f)$,
$$ (f*g)(q)= f(q)g(f(q)^{-1}q f(q)), \qquad \hbox{if} \, \, f(q) \neq 0$$
and $(f*g)(q)=0$ if $f(q)=0$.
\end{prop}
Given a slice regular function $ f: \Omega_D \to \mathbb{H}$ we will use the following notation:
$$ T_f(q):=f(q)^{-1}q f(q)$$
Recall from \cite{GSS, GP}, that given any slice function $f= \mathcal{I}(F) \in \mathcal{S}(\Omega_D)$, then $F^c(z)= F(z)^c=F_{1}(z)^c+ \iota F_2(z)^c$ is a stem function. We define the \emph{slice conjugate} of $f$ as the function $f^c= \mathcal{I}(F^c) \in \mathcal{S}(\Omega_D)$, while the symmetrization of $f$ is defined as $f^s:=f^c * f$. We have that $(FG)^c=G^c F^c$ and so $(f*g)^c=g^c*f^c$, i.e. $f^s=(f^s)^c$.
Notice that, if $f$ is slice preserving, then $f^c=f$ and so $f^s=f^2$.

The theory of slice regular functions has given already many fruitful results, both on the analytic and the geometric side, see for example \cite{ANB, BG, BW3, CS}.
Moreover, slice hyperholomorphic functions have several applications in operator theory and in Mathematical Physics \cite{GMP}.
The spectral theory of the S-spectrum is a natural tool for the formulation of quaternionic quantum mechanics  and for the study of new classes of fractional diffusion problems, see \cite{CSS,CGK,CG},
and the references therein. Slice hyperholomorphic functions are also important in operator theory and Schur analysis which have also been deeply investigated in the recent years, see \cite{ACS,ACS1} and the references therein.

Now, we introduce the following definitions of complex dynamics.

Let $p(z)$ be a complex polynomial.
\begin{Def}
\label{rr2}
The complex filled Julia set is defined as
$$ K_{p}= \{z: (p^{n}(z))_{n} \, \, \hbox{bounded}  \, \,\}.$$
\end{Def}
This set is compact and its boundary  is called the Julia set, i.e. $ \partial K_{p}=J$. The following property will be very useful later.
\begin{Cor}
\label{filled}
Let $p$ be a complex polynomial with real coefficients. Then the filled Julia set is symmetric with respect to the real axis.
\end{Cor}
\begin{proof}
Since $p$ has real coefficients, $p(\bar{z})= \overline{p(z)}.$ Therefore,
$$
p^{n}(\bar{z})= \overline{p^{n}(z)}.
$$
If $z  \in K_{p}$ then $ p^{n}(z)$ is bounded so also $\overline{p^{n}(z)}$ is bounded. From the equality above we have  $p^{n}(\bar{z})$ is bounded and this means $ \bar{z} \in K_{p}$.
\end{proof}
An easy consequence of this Corollary is that the Julia set of a slice preserving polynomial of $\mathbb{H}$ is also symmetric with respect to the real axis, since it is the boundary of the filled Julia set.

Finally, for every $I \in \mathbb{S}$ and every $q \in \mathbb{C}_I$ we set (see \cite[Def. 6.12]{BW})
\begin{equation}
\label{BWL}
(\Delta_*f)(q)=(\Delta_I f)(q),
\end{equation}
where $(\Delta_I)(q)$ is the second order differential operator defined as
$ \Delta_I= \frac{1}{4} \left( \frac{\partial^2}{\partial \alpha^2}+ \frac{\partial^2}{\partial \beta^2} \right)$ and where $f$ is a function of class $\mathcal{C}^2$.
 In \cite[Cor. 6.14]{BW} the authors show that $ \Delta_*$ is well-defined, i.e. that $ (\Delta_I f)(q)=(\Delta_J f)(q)$ for all $I,J \in \mathbb{S}$, $ q \in \mathbb{R}.$

\section{Properties of the operator $\Delta_*$}
In this section we are going to prove some properties related to the operator $\Delta_*$, introduced before. Some of them are very similar to the properties of the bilaplacian, introduced in the papers \cite{P2}, \cite[Prop. \ref{SF} ]{AB}; others are novel and very surprising, see corollary \ref{SFT}. The following corollary of the dominated convergence theorem will be very useful.

\begin{Cor} \label{CD}
If $ \varphi: \mathbb{R}^2 \to \mathbb{R}$ is any positive function such that $ \int_{\mathbb{R}^2} \varphi=1$, then the family of functions depending on $ \varepsilon$,
$$ \varphi_\varepsilon(x)= \frac{1}{\varepsilon^2} \varphi \biggl( \frac{x}{\varepsilon} \biggl),$$
is such that $ \int_{\mathbb{R}^2} \varphi_\varepsilon \equiv 1$, and converges in the sense of distributions to the Dirac delta centered in zero $ \delta_0$ for $ \varepsilon \to 0$.
\end{Cor}

\begin{prop}\label{SF}
The following equality holds in the sense of distributions :
\begin{equation}
\label{lapla1}
\Delta_{*} \bigl(2 \log |q| \bigl)= \delta_{0} \qquad q \in \mathbb{H},
\end{equation}
where $\delta_0$ denotes the Dirac measure centered in zero.
\end{prop}
\begin{proof}
From the formula \eqref{BWL} we have that, for any $I \in \mathbb{S}$:
$$ \Delta_{*} \log|q|= \Delta_{I} \log|q|.$$
First of all, notice that $\log|q|^2$ is a radial function, therefore it is useful to pass to 4D-spherical coordinates
$(r, \theta)=(r, \theta_1, \theta_2, \theta_3)$, where  $r=|q|$. In these coordinates the laplacian is of the form
\begin{equation}
\label{laplap}
\Delta_{I}= \frac{1}{4} \left( \frac{\partial^{2}}{\partial r^{2}}+ \frac{1}{r} \frac{\partial}{\partial r}+ \mathcal{L}( \Theta) \right).
\end{equation}
where $ \mathcal{L}( \Theta)$ is the angular part of the laplacian. In our case this does not give contribution.
Now, we fix $ \varepsilon \in \mathbb{R}$ and we evaluate the laplacian of $ \log(|q|^{2}+ \varepsilon^{2})$ through the formula \eqref{laplap}
$$ \Delta_{I} \log(|q|^{2}+ \varepsilon^{2})= \frac{\varepsilon^{2}}{(r^{2}+ \varepsilon^{2})^{2}}= \frac{1}{\varepsilon^{2} \bigl( \bigl( \frac{r}{\varepsilon} \bigl)^{2}+1 \bigl)^{2}}.$$
If we define $ \varphi(|q|):= \frac{1}{(|q|^{2}+1)^{2}},$ this is an integrable function for which, we can apply the corollary \ref{CD} of the dominated convergence and so
$$\Delta_{I} \log|q|^{2}= \lim_{\varepsilon \to 0} \Delta_{I}\log(|q|^{2}+ \varepsilon^{2})= \delta_{0}.$$
Therefore,
$$ \Delta_{I} \log|q|= \frac{\delta_{0}}{2}.$$
\end{proof}
\begin{Cor}
\label{SFT}
For any $ a \in \mathbb{R}$, the following formulas hold
$$ \Delta_{*} \log|q-a|= \frac{1}{2} \delta_{a},$$
$$ \Delta_{*} \log|q^c-a|= \frac{1}{2} \delta_{a}.$$
\end{Cor}
\begin{proof}
The two formulas follow from the invariance by translation of a real point $a$ of the operator $ \Delta_{*}$ and by Proposition \ref{SF}.
\end{proof}
In order to prove Proposition \ref{CH} we prove the following theorem.
\begin{theorem}
\label{sr}
Let $f: \Omega_{D} \to \mathbb{H}$ be a slice preserving regular function. Then for any $I \in \mathbb{S}$
$$ \Delta_{I} \log|f(q)|=0, \quad \forall q \in \Omega_{D} \setminus  \mathcal{Z}(f).$$
\end{theorem}
\begin{proof}
Since a slice preserving regular function sends each slice in itself and is holomorphic on each slice,
this is a consequence of the analogous holomorphic result for the complex 1-dimensional laplacian and for $\log$ of the modulus of an analytic function.
\end{proof}
\begin{prop}
\label{CH}
For any $ a \in \mathbb{H} \setminus \mathbb{R}$ the following equality holds in the sense of measures :
\begin{equation}
\Delta_{*} \log |(q-a)^{s}|= | \mathbb{S}_{a}|,
\end{equation}
where $|\mathbb{S}_{a}|$ is the Lebesgue measure of the $2$-dimensional real sphere $\mathbb{S}_{a}.$
\end{prop}
\begin{proof}
Let $ a= \alpha_{0}+ I_{0} \beta_{0}$ be any non-real quaternion, since $(q-a)^{s}=q^2-2 \real(a)q +|a|^2$ has real coefficients by Lemma \ref{refee1} is a slice  preserving regular function, for any $q \in \mathbb{H} \setminus \mathbb{S}_{a} $. Thus by Theorem \ref{sr} we get
$$ \Delta_{*} \log |(q-a)^{s}|= \Delta_{I} \log |(q-a)^{s}|=0.$$
Now, we recall that if $h$ is a slice regular function we have
$$ \Delta_{*}h(q_I):=\Delta_{I}h(q_I), \qquad q_I:=q_{| \mathbb{C}_I}.$$
Moreover, since $(q-a)^{s}$ is a slice preserving polynomial we can choose $a \in \mathbb{C}_I$, indeed if $a' \in \mathbb{C}_J$ but it stays in the same sphere of $a$, i.e. $\real(a)= \real{a'}$ and $\Imm (a)= \Imm(a')$ (see \eqref{rr1}), we obtain that
$$(q-a)^{s}=q^2-2 \real(a)q +|a|^2=q^2-2 \real(a') q +|a'|^2=(q-a')^s,$$
the second equality holds because since $a$ and $a'$ stay in the same sphere they have the same real and imaginary parts.
\\So the $*$-product becomes the punctual product since we are working with $q,a \in \mathbb{C}_I$, therefore by complex analysis and Corollary \ref{SFT} we have
\begin{eqnarray}
\Delta_{*} \log |(q-a)^{s}|&=& \Delta_{I} \log |(q-a)^{s}|= \Delta_{I} \log[(q-a)(q-a^c)]\\
\nonumber
&=& \Delta_{I} \log |(q-a)| +  \Delta_{I} \log |(q -a^c)| = \frac{\delta_a}{2}+ \frac{\delta_{\bar{a}}}{2}.
\end{eqnarray}
We denote by $dz_{I}$ and $ d \sigma_{\mathbb{S}}$ the standard surface measures of $ \mathbb{C}_{I} $ and $ \mathbb{S}$, respectively.  If we now take any real valued compactly supported  $C^{\infty}$ function $\varphi$ we have that,
\begin{eqnarray*}
\int_{\mathbb{H}} \Delta_{*} \log|(q-a)^{s}| \varphi(q) \, dq &=& \int_{\mathbb{S}} \biggl( \int_{\mathbb{C}_{I} }
\Delta_{*} \log|(q-a)^{s}| \varphi(q) \, dz_{I} \biggl)d \sigma_{\mathbb{S}}\\
&=& \frac{1}{2} \int_{\mathbb{S}} \varphi( \alpha_{0}+ I \beta_{0})d \sigma_{\mathbb{S}}+ \frac{1}{2} \int_{\mathbb{S}} \varphi( \alpha_{0}- I \beta_{0})d \sigma_{\mathbb{S}}.
\end{eqnarray*}
where the first equality is due to Fubini's theorem and the fact that in a neighbourhood of the real line the integrand is measurable and $\mathbb{R}$ has zero measure with respect the  2-dimensional Lebesgue measure.
\end{proof}

\section{A quaternionic Brolin's theorem: the slice preserving case}
In this section $p$ is a quaternionic polynomial with real coefficients (see Definition \ref{ref2} and Lemma \ref{refee1}).
We consider a sequence of iterates of $p$:
\begin{equation}
\label{este}
p^{n}:= \hbox{ext} (p^{n}_{|_{\mathbb{C}_{I}}}),
\end{equation}
which is independent of $I$. Since the coefficients are real, $p^n$ coincides with the iterates of $p$ on $ \mathbb{H}$, because each slice is sent to itself. We recall that the regular extension of a function $f_I$ is uniquely defined on $\mathbb{H}$, due to the Identity principle. The regular extension is a regular function $g$ defined on $\mathbb{H}$ such that $g_I=f_I$ on $\mathbb{C}_I.$(see \cite[Section 1.3]{GSS}). If $f$ is also slice preserving, then on $ \mathbb{H}$ we have $f=g$, (see \eqref{este}). \\
Now, we redefine the Green's function over the quaternions.
\begin{Def}
\label{def1}
$$ G_{p}(q)_{|\mathbb{C}_{I}} = \lim_{n \to + \infty} d^{-n} \log^{+}|p^{n}(q)|, \quad q \in \mathbb{C}_{I}.$$
We denote
$$G_{n}(q)_{| \mathbb{C}_{I}}= d^{-n} \log^{+}|p^{n}(q)|, \quad q \in \mathbb{C}_{I}.$$
\end{Def}
An equivalent definition is the following
\begin{Def}
\label{def2}
Let $q=x+yJ \in \mathbb{C}_J$, for any $I$ we have
$$ G_{p}(q)= \frac{1}{4 \pi}\int_{I \in \mathbb{S}} G_{p_I}(\tilde{q}) dI, \quad \tilde{q}=x+yI \in \mathbb{C}_{I},$$
where $p_{I}=p_{|_{\mathbb{C}_{I}}}.$
\end{Def}
Recalling the definition of exceptional set, we define $ \mathcal{E}_I$ the exceptional set of the polynomial $p_I$.

Now we are ready to introduce the extension over $\mathbb{H}$ of the Brolin's theorem:
\begin{theorem}
\label{quat}
Let $p$ be a polynomial with real coefficients with degree $d \geq 2$. Then for all $ a \in \mathbb{\mathbb{R}} \setminus (\mathcal{E}_I \cap \mathbb{R})$ the measure
$$ \mu_{n}:= d^{-n} \sum_{b \in p^{-n}(a)} \delta_{b},$$
converges weakly to the equilibrium measure $ \mu$, where $ \delta_b$ is the Dirac delta centred in $b$ if $b$ is a point, otherwise it is the Lebesgue measure of the sphere if $b$ is a sphere. This means that
$$ \mu_{n} \rightharpoonup \mu:= \Delta_{*} G_{p}.$$
The equilibrium measure does not depend on the choice of $a \in \mathbb{\mathbb{R}} \setminus (\mathcal{E}_I \cap \mathbb{R})$.
\end{theorem}

\begin{proof}
Let us first define
\begin{equation}
\mu_{n}^{a}(q):= d^{-n} \log^{+}| p^{n}(q)-a|.
\end{equation}
We know from \cite{B} that $ \mu_{n}^{a}(q)_{|_{\mathbb{C}_{I}}}$ is subharmonic, so by formula \eqref{BWL} we obtain that $ \Delta_{*}\mu_{n}^{a}(q)= \Delta_{I}\mu_{n}^{a}(q) \geq 0.$ This means that the elements in the sequence $ \{ \mu_{n}^{a} \}_{n \in \mathbb{N}}$ are subharmonic .
\\ Now, we are going to prove some important properties of the Green function:
\begin{enumerate}
\item $G_{p}=0$ over $K_{p_{I}} \times \mathbb{S}$ , where $K_{p_{I}}=K_p \cap \mathbb{C}_I$ thus $K_{p_{I}} \subset \mathbb{C}_I$ and so $K_{p_{I}} \times \mathbb{S} \subset \mathbb{H}$. Now, the thesis follows because in $G_n$ we make the product of the limits of an infinitesimal sequence and a bounded one, when $n$ tends to infinity.
\item
\begin{equation}
\label{for1}
G_{p}(p(q))= d \cdot G_{p}(q), \qquad \forall q \in \mathbb{H}.
\end{equation}
Since formula \eqref{for1} holds in every slice, from Definition \ref{def2} we obtain:
$$ G_{p}(p(q))= \frac{1}{4 \pi} \int_{I \in \mathbb{S}}G_{p_I}(p(q)) \, dI= d \cdot \frac{1}{4 \pi} \int_{I \in \mathbb{S}}G_{p_I}(q) \, dI=d \cdot G_{p}(q).$$
\item We can prove the following pointwise convergence
\begin{equation}
\label{convq}
\mu_{n}^{a}(q) \to G_{p}(q), \qquad \forall q \in ( \mathbb{C}_{I} \setminus K_{p_{I}}) \times \mathbb{S}
\end{equation}
\\By definition of limit we have to prove that for all $ \varepsilon >0$ there exists $ \bar{n} \in \mathbb{N}$ such that for any $n > \bar{n}$ we have
$$ | \mu_{n}^{a}(q)-G_{n}(q)| < \varepsilon.$$
Since we are outside of the various filled Julia sets of the slices, we have that $ p^{n}(q)-a \sim p^{n}(q)$, because $p^{n}(q)$ are unbounded.
\begin{eqnarray}
| \mu_{n}^{a}(q)-G_{n}(q)|= \bigl| d^{-n} \log^{+} | p^{n}(q)-a|- d^{-n} \log^{+}| p^{n}(q)| \bigl |= d^{-n} \cdot \biggl| \log^{+} \bigl| \frac{p^{n}(q)-a}{p^{n}(q)} \bigl| \biggl| < \varepsilon.
\end{eqnarray}
\item The Green's function $G_{p}(q)$ is subharmonic and continuous. The continuity follows easily from Definition \ref{def2} and from the fact that on every slice the function $G_{p_I}$ is continuous. Since the Green's function is subharmonic on every slice and from \eqref{BWL} we obtain
$$ \Delta_* G_p(q)= \Delta_{I}(G_p)(q) \geq 0, \qquad q=x+yI \in \mathbb{C}_I.$$
\item In a similar way it follows that $G_{p}(q)$ is harmonic over $(\mathbb{C}_{I} \setminus K_{p_I} )\times \mathbb{S}$, which is a subset of $ \mathbb{H}$ (see property one).
\item At infinity $G_{p}(q) \sim \log|q|$. From the fact that in every slice $G_{p}(z) \sim \log|z|$ and from Definition \eqref{def2} we have
$$ G_{p}(q)=  \frac{1}{4 \pi}\int_{I \in \mathbb{S}} G_{p_I}(q) \sim \frac{1}{4 \pi} \int_{I \in \mathbb{S}} \log |q|=\log|q|, \qquad \forall q \in \mathbb{H},$$
the last equality holds because we are integrating a real quantity over the sphere.
\end{enumerate}

Just to finish off we study the convergence in $L^{1}_{loc}.$
Let $ \{ \mu_{n}^{a}\}_{n \in \mathbb{N}}$ be relatively compact in $L^{1}_{loc}$, by definition this means that there exists a subsequence in $ \{ \mu_{n}^{a}\}_{n \in \mathbb{N}}$ such that we have the following convergence of subsequences in $L^{1}_{loc}$
\begin{equation}
\label{conv2q}
\mu_{n_{i}}^{a}(q) \to \mathcal{V}^{1}.
\end{equation}
Our goal is to prove that $\mathcal{V}^{1}=G_{p}(q)$; in order to reach this aim  we are going to see that $\mathcal{V}^{1}$ satisfies the same properties of the Green's function, seen above.
\\ The function $ \mathcal{V}^{1}$ inherits the properties 2,3,4,5,6 form the convergence $ \mu_{n_{i}}^{a}(q) \to \mathcal{V}^{1}.$ So, we have just to prove that $ \mathcal{V}^{1} \geq 0$ (i.e. the first property). Firstly, we observe that in $(\mathbb{C}_{I} \setminus \{K_{p_I} \}) \times \mathbb{S}$ we have $ \mathcal{V}^{1} > 0$ since $\mu_{n}^{a}(q)$ is positive in every slices.
Inside $K_p=K_{p_I} \times \mathbb{S}$ we have that $ \mathcal{V}^{1} \le 0$ because $p^n(q)$ is a bounded sequence.
Now, we have to prove that $ \mathcal{V}^{1}=0$ on $ \partial K_{p_I} \times \mathbb{S}$. Let $ \xi_{0} \in K_p= K_{p_I} \times \mathbb{S}= \overline{K_{p_I}} \times \mathbb{S},$ so from the continuity of $ \mathcal{V}^{1}$ we have that
$$ 0 \geq \mathcal{V}^{1}(\xi_{0})= \overline{\lim_{ \xi \to \xi_{0}}} \mathcal{V}^{1}(\xi) \geq 0.$$
The last inequality holds because Int($K_{p_I}) \times \mathbb{S}=\varnothing$. To prove this, we suppose by contradiction that Int($K_{p_I}) \times \mathbb{S} \neq \varnothing$. We know that inside $K_{p_I} \times \mathbb{S}$ we have that $ \mathcal{V}^{1}<0$, then there exists $ \beta >0$ such that $ \mathcal{V}^{1}<- \beta$. If $ n_{i}$ is sufficiently large, by \eqref{conv2q} we obtain
$$ \frac{1}{d^{n_{i}}} \log| p^{n_{i}}(q)-a| \to \mathcal{V}^{1},$$
hence
$$  \frac{1}{d^{n_{i}}} \log|p^{n_{i}}(q)-a| <- \frac{\beta}{2}.$$
This means
$$ | p^{n_{i}}(q)-a| < e^{- \frac{\beta d^{n_{i}}}{2}}.$$
However, if there were a decay of order $d^{n_{i}}$, we would have $p(q)=q^{d^{n_{i}}},$ but this is not possible because by assumption $p$ is a generic polynomial, therefore $ \mathcal{V}^{1} \geq 0$ (see page 60, Exercice 2.2.2.22 of Tien Cuong Dinh Cours "Dynamics of holomorphic maps" 2007 available at https:// webusers.imj-prg.fr/~tien-cuong.dinh/Cours2011/Com\_Dyn.pdf.).  Finally, the thesis follows applying the operator $ \Delta_{*}$ to \eqref{conv2q} and  Corollary \ref{SFT} and Proposition \ref{CH}.
\end{proof}

\begin{nb}
It is possible to write the equilibrium measure in the following way $$\mu=  \frac{1}{4 \pi}\int_{I \in \mathbb{S}} \mu_I \, dI,$$
where $\mu_I$ is the equilibrium measure of $p_I$.
\end{nb}

\begin{nb}
If we use the \emph{stem function} (see Definition \ref{Stem}) introduced by Ghiloni and Perotti (see \cite[Paragraph 3]{GP}), the Green's function becomes
$$G_{p}(q)= \lim_{n \to + \infty} \mathcal{I}(d^{-n} \log^{+}|P^{n}(z)|),$$
where $q= \alpha+I \beta$ and $z= \alpha+i \beta$.
From the Corollary \ref{filled} and from the fact that we are working with a polynomial with real coefficients $p$ we have that the Julia set of the \emph{stem function} $P$ is equal to the Julia set of  $p_i:=p_{| \mathbb{C}_I}$. Another proof of Theorem \ref{quat} follows using the complex Brolin's theorem applied to the Green function of the stem $P$
$$G_{p}(z)= \lim_{n \to + \infty} d^{-n} \log^{+}|P^{n}(z)|.$$
\end{nb}

\section{A quaternionic Brolin's theorem: the one slice preserving case}
Let $\mathcal{P}_I$ be a polynomial with degree $d \geq 2$ and all coefficients in $\mathbb{C}_I$ but not all real, these kind of polynomials are called one slice preserving (see Definition \ref{ref5}). In this case
the idea is to evaluate, at any $n$-th iterate, the symmetrized function. This means that for every $n$ we are making the following operation :
\begin{equation}\label{gn}
\mathcal{P}^{n}_{I} * (\mathcal{P}^{n}_{I})^{c}=(\mathcal{P}^{n}_{I})^{s}:= g_n.
\end{equation}
Obviously, if the polynomial has real coefficients from the above formula we obtain $(\mathcal{P}^{n}_{I})^{2}=(\mathcal{P}^{n}_{I})^s$. Moreover, we can remark that in general the $g_n$ does not follow any iteration relation, i.e.
$$ g_{n+1} \neq g_1 (g_n).$$
We know over the complex that if $q \in \mathbb{C}_I$ is a root of $\mathcal{P}_I(q) $ then $q^c \in \mathbb{C}_I$ is a root of $\mathcal{P}_I^c(q)$. This means that both $q$ and $q^c$ are roots of $g_n$.
\begin{lemma}
\label{Par}
Let $\mathcal{H}_I$ be a polynomial with all coefficients in the slice $\mathbb{C}_I$ and X a closed set in $ \mathbb{R}$ with positive equilibrium measure then all coefficients of $\mathcal{H}_I$ have to be real.
\end{lemma}
\begin{proof}
Since the support of the measure is the Julia set we have that it is contained in a finite union of iterations of the set $X$. In particular if $X$ is a smooth path in $\mathbb{R}$ then the Julia set is contained in  a finite union of smooth paths. So the hypotheses of \cite[Theorem 5, chapter 5]{S} are satisfied, thus we obtain that the polynomial is conjugate to $z^d$ or to the Chebyshev polynomials. Hence $\mathcal{H}_I$ has real coefficients.
\end{proof}
Denoting $ \mathcal{E}'_I$ the exceptional set of $g_n:=(\mathcal{P}^{n}_{I})^{s}$ restricted to $ \mathbb{C}_I$, which cardinality is less or equal than two, in the sense that we can have at most two among points and spheres.
Now we can prove the Brolin's theorem for the one slice preserving case.

\begin{theorem}
\label{Brolin2}
Let $\mathcal{P}_I$ be a polynomial with degree $d \ge 2$ and all coefficients in $\mathbb{C}_I$ but not all real. Then we have for all $a \in \mathbb{R} \setminus ( \mathcal{E}'_I \cap \mathbb{R}):$
$$  \frac{1}{2}\Delta_* \log^{+} |g_{n}-a| \rightharpoonup \frac{1}{8 \pi} \int_{I \in \mathbb{S}} \mu_{\mathfrak{P}(I)} \, dI + \frac{1}{8 \pi} \int_{I \in \mathbb{S}} \mu_{\mathfrak{P}^{c}(I)} \, dI:= \mu',$$
where if $\mathcal{P}_I(q)=q^d a_d(I)+...+a_0(I)$ with $ q \in \mathbb{H}$ and $a_{\ell}(I)=x_\ell+I y_\ell$ ($x_\ell, y_\ell \in \mathbb{R}$) for $ \ell =0,...,d$, then, for all $J \in \mathbb{S}$, we have $ \mathfrak{P}(q, J)=q^da_d(J)+...+a_0(J)$ with $ q \in \mathbb{C}_J$ and $a_{\ell}(J)=x_\ell+J y_\ell$ ($x_\ell, y_\ell \in \mathbb{R}$) for $ \ell =0,...,d$.
We call the limit $ \mu'$ \emph{the Equilibrium measure of $\mathcal{P}_I$} in this case and it does not depend on the chosen $a \in \mathbb{R}\setminus (\mathcal{E}'_I \cap \mathbb{R}).$
\end{theorem}
\begin{proof}
Since we are working with the functions $g_n$ which are slice preserving, we have by Theorem \ref{quat}:
\begin{equation}
 \frac{1}{2}\lim_{n \to \infty} \Delta_* \log^{+} |g_{n}-a|= \frac{m_i}{2} \sum_{n=1}^\infty \delta_{\mathcal{Z}_{\mathfrak{P}^n(q,I)}}+  \frac{k_i}{2} \sum_{n=1}^\infty \delta_{\mathcal{Z}_{(\mathfrak{P}^n(q,I))^c}},
\end{equation}
where $\mathfrak{P}^n(q,I)$ is the iterate of $\mathfrak{P}(q, I)$, the set $ \mathcal{Z}$ is defined in \eqref{refee} and $m_i$ and $k_i$ are the multiplicity of the roots of $\mathcal{P}^{n}_{I}$ and $(\mathcal{P}^{n}_{I})^c$, respectively.
Finally, from the complex version of Brolin's theorem we obtain
\begin{eqnarray*}
\frac{1}{2}\lim_{n \to \infty} \Delta_* \log^{+} |g_{n}-a|&=& \frac{m_i}{2} \sum_{n=1}^\infty \delta_{\mathcal{Z}_{\mathfrak{P}^n(q,I)}}+ \frac{k_i}{2} \sum_{n=1}^\infty \delta_{\mathcal{Z}_{(\mathfrak{P}^n(q,I))^c}}\\
& = &   \frac{1}{8 \pi} \int_{I \in \mathbb{S}} \mu_{\mathfrak{P}(I)} \, dI + \frac{1}{8 \pi} \int_{I \in \mathbb{S}} \mu_{\mathfrak{P}^{c}(I)} \, dI:= \mu'.
\end{eqnarray*}
We pointed out that $I$ in $\mathcal{P}_I(q)$ is fixed, while in $\mathfrak{P}(q, J)$ the imaginary unit $J$ is a variable.
Due to the exclusion of the polynomials with all real coefficients we have that the equilibrium measure is well defined over the reals since the measure has no mass on it (see Lemma \ref{Par}).
\end{proof}
This calculus fits well if the polynomial has real coefficients, indeed in this case
\begin{Cor}
If $p$ is a polynomial with real coefficients of degree $d \ge 2$, then for any $a \in \mathbb{R} \setminus (\mathcal{E}_I \cap \mathbb{R}):$
$$ \frac{1}{2}\Delta_* \log^{+} |g_{n}-a| \rightharpoonup \frac{1}{4 \pi} \int_{I \in \mathbb{S}} \mu_{\mathfrak{P}(I)} \, dI.$$
\end{Cor}

\begin{nb}
Now, one can ask if there exist a sequence $ \{ h_j \}_{j \in \mathbb{N}}$ of polynomials with real coefficients such that
\begin{equation}
\label{obs}
h^{m}_{j}=g_n,
\end{equation}
where $g_n$ is the sequence defined previously.

Let $ \hbox{deg} \,h_1=e \in \mathbb{N}$ from \eqref{obs} we obtain
\begin{equation}
\label{rela}
e^m=2d^n.
\end{equation}
This relation is impossible since if we consider $e=2$ and $d$ a power of $2$, choosing a generic polynomial we realize that the sequence $ \{h_j\}_j$ has always an extra monomial with respect to the sequence $g_n$.
\end{nb}
\begin{nb}
If we had used the slice regular extension \eqref{este} instead of its symmetrized to prove Theorem \ref{Brolin2}  we would have had some problems. Indeed, for instance, choosing a monomial $ f(q)= q^m j,$ with $m \in \mathbb{N}$, the existence of the equilibrium measure would depend on the choice of the point on $ \mathbb{C}_j$.
\end{nb}

\begin{nb}
If we evaluate the stem of $g_n$ by \cite[Prop. 12]{GP} we have
$$ \mathcal{I}(\mathcal{P}^{n}_{I})* \mathcal{I}((\mathcal{P}^{n}_{I})^c)= \mathcal{I}(\mathcal{P}^{n}_{I}  \cdot (\mathcal{P}^{n}_{I})^c) := \mathcal{I}(CN(\mathcal{P}^{n}_{I} )).$$
Now we consider
$$ \widetilde{G}:= \lim_{n \to \infty} d^{-n} \bigl(\mathcal{I}(\log^{+}|CN(\mathcal{P}^{n}_{I})|) \bigl).$$
Therefore we have
$$ \Delta_* \widetilde{G}= \mu'.$$
\end{nb}

\section{A quaternionic Brolin's theorem: the general case}
In this section we are going to discuss what happens to the Brolin's theorem when we consider a slice regular polynomial with coefficients in different slices with degree $d \geq 2$ . We cannot use the idea of the previous case since now we cannot fix a slice where Brolin's theorem holds furthermore the composition (as in \eqref{este}) of polynomials with generic coefficients does not give a slice regular polynomial. So in order to overcome this problem we consider the following composition introduced in \cite{GSS1, CSS1}
\begin{Def}
Denoting $g(q)= \sum_{n=0}^{\infty} q^n a_n$ and $ w(q)=\sum_{n=0}^{\infty} q^n b_n$. We define
$$ (g \bullet w)(q)= \sum_{n=0}^\infty (w(q))^{*n} a_n,$$
where $(w(q))^{*n}$ means that we take the $n$-th power with respect to the $*$- product.
\end{Def}
\begin{nb}	
When we compose in this way the degree of the polynomial grows exponentially since the degree of the composition is the product of the degrees of the starting polynomials.
\end{nb}
Now we denote the $n$-th composition of a generic polynomial $p(q)$ as $p^{\bullet n}$ and for every $n$ we make the symmetrized:
$$p^{\bullet n} *(p^{\bullet n})^c=(p^{\bullet n})^s:= h_n.$$
We introduce the orbit as
$$ [q]:= \{h_n(q)| \, \, \, n \in \mathbb{N}\}.$$
\begin{Def}
The exceptional set for this kind of polynomials is defined as
$$ \mathcal{E''}:= \{q: \, |[q]|< \infty \}.$$
\end{Def}

\begin{theorem}
\label{Brolin3}
Let $p(q)$ be a slice regular polynomial with degree $d \ge 2$ and coefficients in different slices. Let $h_n  =(p^{\bullet n})^{s}$. Then, for all $a \in \mathbb{R} \setminus (\mathcal{E''} \cap \mathbb{R})$ and $b \in \mathbb{R} \setminus (\mathcal{E''} \cap \mathbb{R})$ with $a \neq b$  we have that
$$ d^{-n} \log|h_n(q)-a|-d^{-n} \log|h_n(q)-b| \to 0.$$
\end{theorem}
\begin{proof}
We want to use \cite[Thm. 4.6]{DS3}. The hypotheses of the theorem are verified, indeed choosing $ \delta_{n}=1$ the series $\sum_{n=1}^\infty (d^n)^{-1}$ is a geometric series with ratio less than one so it converges and moreover by the definition of pull-back we have:
$$ d^{-n} \log|h_n(q)-a|-d^{-n} \log|h_n(q)-b| \to 0.$$
\end{proof}
Theorem \ref{Brolin3} is not a Brolin's theorem as Theorem \ref{Brolin2} and Theorem \ref{quat} because in those cases at least in a slice we have the real iterates of the polynomial while in this case the iterate of the polynomial is "artificial". Moreover Theorem \ref{Brolin3} does not give a limit measure as in the previous cases.

\section{Mixing and rate of mixing}
Let us consider $p(q)= \sum_{j=1}^{d} a_{j}q^{j}$, with every $a_j \in \mathbb{R}$. We want to solve $p(q)=b$, $b \in \mathbb{R} \setminus (\mathcal{E'} \cap \mathbb{R})$. In order to do this we denote $ \mathbb{S}^{n}_{j}$, with $j=1,..., d^n$  the sets such that
$$ p^{n}(\mathbb{S}^{n}_{j})=b.$$
\begin{Def}
\label{delta}
$$ \delta_{\mathbb{S}^{n}_{j}}=
\begin{cases}
\delta_{r} \quad \hbox{if} \quad \mathbb{S}^{n}_{j} \quad \hbox{is a real point r}\\
\frac{1}{2} \int_{I \in \mathbb{S}} \sum \delta_{(\mathbb{S}^{n}_{j} \cap \mathbb{C}_{I})} \, dI
\end{cases}
$$
\end{Def}
We observe that if $ \mathbb{S}^{n}_{j}$ is a sphere $ \mathbb{S}_{\alpha+I \beta}$, then the intersection $\mathbb{S}^{n}_{j} \cap \mathbb{C}_{I}$ is given by two points of the following form $ \alpha+I \beta$ and $ \alpha-I \beta$.
\\Now, we prove in another way Theorem \ref{quat}.
\begin{theorem}
The measure
$$ \nu_{n}= \frac{1}{4 \pi d^{n}} \sum_{j=1}^{d^n} \delta_{\mathbb{S}^{n}_{j}},$$
converges to the equilibrium measure $ \mu= \frac{1}{4 \pi} \int_{I \in \mathbb{S}} \mu_{I} dI$.
\end{theorem}
\begin{proof}
We have to prove that if $ \varphi \in C^{0}(\mathbb{H})$ then
$$ \langle \nu_{n}, \varphi \rangle \to \langle \mu, \varphi \rangle.$$
By definition of $ \nu_n$ we get
\begin{eqnarray}
\label{me1}
\langle \nu_{n}, \varphi \rangle  \! \! \! \! \! \! \! \!&& = \int_{I \in \mathbb{S}} \varphi \biggl( \frac{1}{4 \pi d^{n}} \sum \delta_{\mathbb{S}^{n}_{j}} \biggl) \, dI\\
&&= \int_{I \in \mathbb{S}}  \frac{1}{4 \pi d^{n}} \sum_{p^{n}(a_{j})=b} \varphi(a_{j}) \, dI= \int_{I \in \mathbb{S}} \frac{p^{*n}(\varphi)}{4 \pi d^{n}} \, dI, \nonumber
\end{eqnarray}
in the second equality we used Definition \ref{delta}, while in the last we have used the definition of pull-back. From the complex version of the Brolin's theorem we know that for $q \in \mathbb{C}_I$
\begin{equation}
\label{me2}
d^{-n}(p^{*})^{n} \delta_q \to \mu_{I}.
\end{equation}
So from \eqref{me1} and \eqref{me2} we obtain the thesis.
\end{proof}
\begin{Def}
Over the complex numbers the set of critical points of $p$ is defined in the following way
$$ C_{\mathbb{C}_{I}}:= \{ z \in \mathbb{C}: p'_{I}(z)=0 \}.$$
While, over the quaternions
$$C_{\mathbb{H}}:= \{ q \in \mathbb{H}: \partial_{c} p(q)=0 \},$$
We remark that $ C_{\mathbb{H}}= \bigcup_{I \in \mathbb{S}} C_{\mathbb{C}_I}.$
\end{Def}
\begin{lemma}
The equilibrium measure of the quaternionic critical points is zero, i.e.
$$ \mu(C_{\mathbb{H}})=0.$$
\end{lemma}
\begin{proof}
It follows easily from the fact that $ \mu_{I}(C_{\mathbb{C}_I})=0$:
$$  \mu(C_{\mathbb{H}})= \frac{1}{4 \pi} \int_{C_{\mathbb{C}_I}} \mu_{I} \, dI= \frac{1}{4 \pi} \mu_{I}(C_{\mathbb{C}_I})=0.$$
\end{proof}
\begin{prop}
\label{inv}
\begin{enumerate}
\item $p_{*}(\mu)=\mu,$
\item $p^{*}(\mu)=d \cdot \mu$,
where $d$ is the degree of the polynomial
\end{enumerate}
\end{prop}
\begin{proof}
\begin{enumerate}
\item
Let us consider $ \varphi \in C^{0}(\mathbb{H})$,
\begin{eqnarray*}
p_{*}(\mu)(\varphi)= \frac{1}{4 \pi} \int_{I \in \mathbb{S}} \varphi d(p_{*} \mu_{I})= \frac{1}{4 \pi} \int_{I \in \mathbb{S}} ( \varphi \circ p_{I}) \, d \mu_{I}= \frac{1}{4 \pi} \int_{I \in \mathbb{S}} \varphi \, d \mu_{I}= \mu( \varphi).\\
\end{eqnarray*}
where in the second equality we used a well-known property of the derivative of the push-forward, whereas in the third inequality we used the fact that the complex polynomials are surjective, i.e. $p(\mathbb{C}_{I})=\mathbb{C}_{I}$.

\item Let $E$  \emph{axially symmetric} subset of $\mathbb{H}$, from \eqref{for1} we obtain
\begin{eqnarray*}
p^{*}( \mu)(E)& =&\frac{1}{4 \pi} \int_{I \in \mathbb{S}}(p^{*} \mu_{I})(E_{I}) \, dI= \frac{1}{4 \pi} \int_{I \in \mathbb{S}} \mu_{I}(p_{*} \chi_{E_{I}}) \, dI\\
&=& \frac{d}{4 \pi} \int_{I \in \mathbb{S}} \mu_{I}( \chi_{E_{I}}) \, dI=\frac{d}{4 \pi} \int_{I \in \mathbb{S}} \mu_{I}(E_{I}) \, dI=d \cdot \mu(E).
\end{eqnarray*}
\end{enumerate}
\end{proof}
\begin{Def}
A measure is called mixing if given $ \varphi, \psi \in C^{0}(\mathbb{H})$
$$ \langle \mu, ( \psi \circ p^{n}) \cdot \varphi \rangle \to \langle \mu, \varphi \rangle \langle \mu, \psi \rangle.$$
\end{Def}
\begin{lemma}
The equilibrium measure, in the slice preserving case, is mixing.
\end{lemma}
\begin{proof}
Let $ \varphi \in C^{0} (\mathbb{H})$, then $d^{-n}(p^{n}_{I})^{*} \delta_{q} \to \mu_{I}$, with $ q \in \mathbb{C}_I$. By duality
\begin{equation}
\label{bi1}
d^{-n}(p^{n}_{I})_{*} \varphi \to \int_{\mathbb{C}_{I}} \varphi \, d \mu_{I}.
\end{equation}
Let $\psi$ be another function in $C^{0}(\mathbb{H})$, then
\begin{eqnarray*}
\langle \mu, ( \psi \circ p^{n}) \cdot \varphi \rangle &=& \frac{1}{4 \pi}  \int_{I \in \mathbb{S}}\int_{\mathbb{C}_{I}} \biggl( \varphi \cdot ( \psi \circ p^{n}_{I}) \, d \mu_{I} \biggl) \, dI\\
& =& \frac{1}{4 \pi}  \int_{I \in \mathbb{S}} \int_{\mathbb{C}_{I}} \biggl( \varphi (\psi \circ p^{n}_{I}) \, d \bigl( \frac{(p^{n}_{I})^{*}}{d^{n}} \mu_{I} \bigl) \biggl) \, dI\\
&=&  \frac{1}{4 \pi}  \int_{I \in \mathbb{S}} \int_{\mathbb{C}_{I}} \biggl( \varphi  \, d \bigl( \frac{(p^{n}_{I})(\psi\mu_{I})}{d^{n}}  \bigl) \biggl) \, dI\\
&=&  \frac{1}{4 \pi} \int_{I \in \mathbb{S}} \int_{\mathbb{C}_{I}} \biggl( \frac{(p^{n}_{I})_{*} \varphi}{d^{n}} \psi  \, d  \mu_{I} \biggl) \, dI.\\
\end{eqnarray*}
where in the second equality we used the second part of Proposition \ref{inv}. Now, using \eqref{bi1} and the mixing property in the $ \mathbb{C}_I$ integral we obtain that
$$  \frac{1}{4 \pi}  \int_{I \in \mathbb{S}} \int_{\mathbb{C}_{I}} \biggl( \frac{(p^{n}_{I})_{*} \varphi}{d^{n}} \psi  \, d  \mu_{I} \biggl) \, dI \to \frac{1}{4 \pi} \int_{I \in \mathbb{S}} \biggl( \int_{\mathbb{C}_{I}} \varphi \, d \mu_{I} \biggl) \biggl( \int_{\mathbb{C}_{I}} \psi \, d \mu_{I} \biggl)=\langle \mu, \varphi \rangle \langle \mu, \psi \rangle.$$
\end{proof}
\begin{prop}
The equilibrium measure $\mu$, in the slice preserving case, is exponentially mixing with decay of correlation $d^{-n}$, that is for $ \varphi \in L^{\infty}(\mu)$ and $ \psi \in C^{2}(\mathbb{H})$ we have
\begin{equation}
\frac{1}{4 \pi} \int_{I \in \mathbb{S}}| \langle \mu_{I}, ( \varphi \circ p^{n}) \psi \rangle- \langle \mu_{I}, \varphi \rangle \langle \mu_{I}, \psi \rangle| \, dI \leq B \| \varphi \|_{\infty} \| \psi \|_{C^{2}(\mathbb{H})} d^{-n},
\end{equation}
where $B$ is a positive constant.
\end{prop}
\begin{proof}
Let $ \psi \in C^{2}(\mathbb{H})$ and $ \varphi \in L^{\infty}(\mathbb{H})$. On each slice, according to \cite[Thm. 2.2.1.12]{DS}(we use the facts that the norm $L^1$ is less or equal than the DSH norm and $ \|.\|_{DSH} \leq c \|.\|_{C^{2}}$, see \cite[Exercise 2.2.4.3]{DS}), we have that:
$$ \| \Lambda^{n} \psi-c_{\psi} \|_{L^{1}(\mu_{I})} \leq B \| \psi \|_{C^{2}(\mathbb{H})}d^{-n},$$
where $ \Lambda^{n}:= d^{-n}(p^{n}_{I})_{*}$ and $ c_{\psi}:= \langle \mu_{I}, \psi \rangle$. The constant $B$ is the same for each slice $ \mathbb{C}_I$. Then
\begin{eqnarray*}
\frac{1}{4 \pi} \int_{I \in \mathbb{S}} | \langle \mu_{I}, (\varphi \circ p^{n}) \psi \rangle- \langle \mu_{I}, \varphi \rangle \langle \mu_{I}, \psi \rangle| \, dI &=& \frac{1}{4 \pi} \int_{I \in \mathbb{S}}  \left| \int_{\mathbb{C}_{I}} \varphi(p^n) \psi d u_I-\int_{\mathbb{C}_{I}} \varphi d \mu_{I}  \int_{\mathbb{C}_{I}} \psi  d \mu_{I} \right| dI\\
&\leq& \frac{1}{4 \pi} \int_{I \in \mathbb{S}} \left( \int_{\mathbb{C}_{I}} |\varphi (\Lambda^n \psi-c_{\psi})|d \mu_{I} \right) dI\\
& \leq &  \frac{1}{4 \pi} \int_{I \in \mathbb{S}} \, dI  \, \,\| \varphi \|_{\infty}  \| \Lambda^{n} \psi-c_{\psi} \|_{L^{1}(\mu_{I})}\\
&=&\| \varphi \|_{\infty}  \| \Lambda^{n} \psi-c_{\psi} \|_{L^{1}(\mu_{I})}\\
&\leq &B \| \varphi \|_{\infty} \| \psi \|_{C^{2}(\mathbb{H})}d^{-n}.
\end{eqnarray*}
\end{proof}
\begin{nb}
In the one slice preserving case we can obtain all the results of this section using $g_n$, which is slice preserving.
\end{nb}

\section{A central limit theorem: slice preserving case}
In this section we want to prove a central limit theorem over the quaternions.  Before starting we recall the classical  result
\begin{theorem}
Let $\mu$ be the equilibrium measure in the slice preserving case. We consider $ \varphi: \mathbb{H} \to \mathbb{R}$ and $f:\mathbb{H} \to \mathbb{H}$. If $\varphi \in L^{1}(\mu)$, $\varphi$ is not a coboundary (see \cite[pag. 211]{DS4}) and $f$ is a measurable map, then $ \frac{1}{\sqrt{n}} \sum_{i=0}^{n-1} \varphi \circ f^i$ converges in law to a Gaussian random variable of zero mean and variance $\sigma>0$. That is, for every interval $I \subseteq \mathbb{R}$
$$ \lim_{n \to + \infty} \mu \biggl \{ \frac{1}{\sqrt{n}} \sum_{i=0}^{n-1} \varphi \circ f^i \in I \biggl \}= \frac{1}{\sqrt{2 \pi} \sigma} \int_{I} e^{- \frac{x^{2}}{2 \sigma^2}} \, dx.$$
\end{theorem}

In our context $f$ is a quaternionic polynomial with real coefficients. In order to prove this theorem we recall a result of Dinh and Sibony.
In \cite{DS2},  they prove in the more general setting of compact K\"ahler manifolds of dimension $k$, the following theorem that we restate in dimension 1, for sake of completeness.
\begin{theorem}
The equilibrium measure $\mu$ is exponentially mixing in the following sense: for every $\varepsilon>0$, there exists a constant $A_\varepsilon>0$ such that
\begin{equation}
\label{R21}
|\langle \mu, (\varphi \circ f^n) \psi \rangle- \langle \mu, \varphi \rangle \langle \mu, \psi \rangle| \leq A_\varepsilon \| \varphi \|_{\infty} \| \psi \|_{Lip}(1+ \varepsilon)^{\frac{n}{2}}d^{-\frac{n}{2}}
\end{equation}
for all $n \geq 0$, every function $ \varphi \in L^{\infty}(\mu)$, which goes from $ \mathbb{H}$ to $ \mathbb{R}$, and every lipschitzian function $\psi$, which goes from $ \mathbb{H}$ to $ \mathbb{R}$. In particular, if a real-valued lipschitzian function $\psi$ is not a coboundary and verifies $ \langle \mu, \psi \rangle=0$ then it satisfies the CLT.
\end{theorem}
In paper \cite{DS2} this theorem implies an estimate on the Perron-Frobenius operator $\Lambda$ such that
\begin{equation}
\label{in1}
\sum_{n \geq 0} \| \Lambda^{n} \psi \|_{L^{1}(\mu)}<+ \infty.
\end{equation}
\begin{prop}
The central limit theorem holds for the Lipschitz function $\psi$ in the quaternionic setting.
\end{prop}
\begin{proof}
Since the Quaternionic Equilibrium Measure in the slice preserving case constructed in this paper is a mean on a 2-dimensional sphere of the 1-dimensional complex equilibrium measure $\mu_I$ which satisfies the inequality \eqref{R21} with the same constant $A_\varepsilon$ on every slice $\mathbb{C}_I$, then, we deduce the same decay of correlation $d^{-\frac{n}{2}}$ when $\psi$ is lipschitzian on $\mathbb{H}$ and the same estimate for the Perron-Frobenious operator $\Lambda$ associated to quaternionic polynomials with real coefficients; via the  Gordin-Liverani theorem (\cite{G}, \cite[Thm. 1.1]{L}). This implies that $\psi$ satisfies the CLT.
\end{proof}

\section{Topological entropy, measurable entropy and Lyapounov exponents}
The entropy of a measure introduced by Kolmogorov and the topological entropy are fundamental notions in dynamics. We follow Bowen's approach for the topological entropy. Let $n$ be a positive integer. For any $\varepsilon$ positive number we consider $p$ a quaternionic polynomial with real coefficients.
\begin{Def}
\label{sep}
Two points $ q_{1}, q_{2} \in \mathbb{H}$ are $(n , \varepsilon)$-separated if there is an integer
$j$ such that $0 \leq j \leq n-1$ and $|p^{j}(q_{1})-p^{j}(q_{2})| \geq \varepsilon.$
\end{Def}
The last inequality means that the distance between $p^{j}(q_{1})$ and $p^{j}(q_{2})$ is larger or equal to $\varepsilon$, before time $n$.
\begin{Def}
\label{num}
Let $K$ be a compact axially symmetric subset of $\mathbb{H}$. We denote with $N (K, n, \varepsilon)$ the maximal number of points in $K$ which are pairwise $ (n, \varepsilon)$- separated.
\end{Def}
\begin{nb}
\begin{enumerate}
\item $ N (K, n, \varepsilon) \geq 1$.
\item If $\varepsilon$ decreases to zero, $N (K, n, \varepsilon)$ increases.
\end{enumerate}
\end{nb}
\begin{Def}
\label{entropy}
The topological entropy of $p$ on $K$ is
\begin{equation}
h_{t}(p,K)= \sup_{\varepsilon >0} \overline{\lim_{n \to + \infty}} \frac{1}{n} \log N(K,n, \varepsilon),
\end{equation}
\end{Def}
As we observed above $ N (K, n, \varepsilon) \geq 1$, so the definition is well-posed.

For $n \geq 0$ and $q_1, q_2 \in \mathbb{H}$, we define $\dis_n(q_1,q_2)$ as
$$d^p_n:=\dis_n (q_1,q_2):= \max_{0 \leq j \leq n-1} |p^{j}(q_1)-p^{j}(q_2)|$$
\begin{Def}[Bowen's ball]
A  $(n, \varepsilon)$-Bowen's ball with centre $a \in \mathbb{H}$ is defined in the following way
$$ \mathbb{B}(a;n, \varepsilon):=\{q \in \mathbb{H} : \dis_{n}(q,a)< \varepsilon \}$$
\end{Def}

This definition is the opposite of the definition of $(n, \varepsilon)$-separated points.
Now, we give an intuitive proof of the following Lemma.
\begin{lemma}
$ N (K, n, \varepsilon)$ is the maximum number of  $ \mathbb{B}\bigl(a;n, \frac{\varepsilon}{2}\bigl)$ centred in points of $K$.
\end{lemma}
\begin{proof}
Firstly, we consider inside the compact axially symmetric $K$ two tangent Bowen's balls of radii $ \frac{\varepsilon}{2}$, so the distance among their centres is $ \varepsilon$. Therefore, if we consider two non tangent balls we have that the distance between the centres is at least $ \varepsilon$.
\end{proof}

Now, we define another important number.
\begin{Def}
Let $K$ be a compact axially symmetric subset of $\mathbb{H}$. We denote with $M (K, n, \varepsilon)$ the minimal number of  $(n,\varepsilon)$-Bowen's balls to cover $K$.
\end{Def}
We can estimate $N(K,n , \varepsilon)$ from above and below with this number:
\begin{equation}
\label{stima}
M(K, n, 2 \varepsilon) \leq N(K,n, \varepsilon) \leq M(K,n, \varepsilon /2).
\end{equation}
From this estimate it is possible to restate Definition \ref{entropy} in the following way
\begin{def}
\label{entropy1}
\begin{equation}
h_{t}(p,K)= \sup_{\varepsilon >0} \overline{\lim_{n \to + \infty}} \frac{1}{n} \log M(K,n, \varepsilon).
\end{equation}
\end{def}
Now, we want to extend the following theorem in the quaternionic context.
\begin{theorem}
\label{Poen}
Let $p: \mathbb{C} \to \mathbb{C}$ a complex polynomial with degree $d$, the topological entropy is given by $\log d$.
\end{theorem}
\begin{proof}
\cite[Thm. 1.1.5.11]{DS}
\end{proof}

Now, we define the entropy of a measure. Let $\nu$ be a probability measure
with compact support. Let $ \xi= \{ \mathbb{S} \times A_1  ,... ,\mathbb{S } \times A_m  \}$ be a finite partition of $K$ compact axially symmetric subsets of $\mathbb{H}$, where $A_i$ are pairwise disjoint complex Borel sets such that $ \nu \bigl(( \mathbb{S } \times A_1 ) \cup... \cup( \mathbb{S }\times A_m ) \bigl)=1$.
\begin{Def}
\label{entro}
The entropy of the partition $ \xi$ with respect to $\nu$ is the following non-negative number
$$ H(\nu, \xi)=- \sum_{C \in \xi} \nu(C) \log(\nu(C)).$$
\end{Def}
We observe that this notion is independent of the polynomial $p$.

Let $ \xi_{n}^p$ denote the finite partition $ \xi \vee p^{-1}(\xi)\vee...\vee p^{-n+1}(\xi)$, i.e. the partition formed by the disjoint Borel sets of type
$$ (A_{i_0} \times \mathbb{S}) \cap p^{-1}(A_{i_1} \times \mathbb{S}) \cap... \cap p^{-n+1}(A_{i_{n-1}} \times \mathbb{S})$$
for $1 \leq i_0,...,i_{n-1} \leq m$.
We define
$$ h_{p}(\nu, \xi)=\lim_{n \to \infty} \frac{1}{n}H (\nu, \xi_n^p).$$

\begin{Def}
The entropy of the invariant probability measure $\nu$ is
$$ h_p (\nu)= \sup_{\xi}h_{p}(\nu, \xi).$$
\end{Def}
Now we prove that the entropy of the probability measure $ \nu$ is affine \cite[Thm. 8.1]{W}.
\begin{lemma}
\label{aff}
Let $ \nu_1$, $ \nu_2$ be probability measures and $ \alpha \in [0,1]$ then
 $$h_p( \alpha \nu_1+(1-\alpha) \nu_2)= \alpha h_p(\nu_1)+(1-\alpha)h_p( \nu_2)$$
\end{lemma}
\begin{proof}
Firstly let us denote $ \phi(x)= x \log x $, we remark that this function is convex. Let $B$ be a quaternionic Borel subset
\begin{eqnarray*}
0 & \geq& \phi [ \alpha \nu_1(B)+(1- \alpha)\nu_2(B) ]-  \alpha \phi(\nu_1(B))-(1-\alpha) \phi(\nu_2(B))\\
&=& (\alpha \nu_1(B)+(1- \alpha)\nu_2(B)) \log [ \alpha \nu_1(B)+(1- \alpha)\nu_2(B) ]\\
&& -\alpha \nu_1(B) \log \nu_1(B)-(1-\alpha)\nu_2(B) \log \nu_2(B)=\\
\end{eqnarray*}
Adding and subtracting the terms $ \alpha \nu_1(B) \log( \alpha \nu_1(B))$ and $(1-\alpha) \nu_2(B) \log[(1- \alpha)\nu_2(B)]$ and using the fact that $\log$ is increasing we have
\begin{eqnarray*}
&& =\alpha \nu_1(B)  \bigl[\log [ \alpha \nu_1(B)+(1- \alpha)\nu_2(B)]- \log (\alpha \nu_1(B)) \bigl]
+(1- \alpha) \nu_2(B) \bigl[\log [ \alpha \nu_1(B)+(1- \alpha)\nu_2(B)]\\
&& - \log \bigl((1- \alpha) \nu_2(B)\bigl) \bigl] + \alpha \nu_1(B)[\log (\alpha \nu_1(B))- \log(\nu_1(B))]+ (1- \alpha) \nu_2(B)[\log \bigl((1- \alpha) \nu_2(B)\bigl)\\
&& -\log (\nu_2(B))] \geq 0+0+\alpha \nu_1(B) \log \alpha+ (1- \alpha) \nu_2(B) \log (1- \alpha).\\
\end{eqnarray*}
If $\xi$ is a generic partition, by Definition \ref{entro}, we obtain:
$$ 0 \leq H(\alpha \nu_1+(1- \alpha)\nu_2, \xi)- \alpha H( \nu_1, \xi)-(1- \alpha)H(\nu_2, \xi) \leq -\bigl(\alpha  \log \alpha+ (1- \alpha) \log (1- \alpha) \bigl) \leq \log 2.$$
If $ \eta$ is a finite partition putting $ \xi= \eta \vee p^{-1}(\eta) \vee...\vee p^{-n+1}(\eta)$, multiplying by $1/n$ and making the limit in the above formula we obtain:
\begin{equation} \label{entro1}
h_p(\alpha \nu_1+(1- \alpha)\nu_2, \eta)=\alpha h_p(\nu_1, \eta)+(1-\alpha)h_p( \nu_2, \eta).
\end{equation}
This implies
$$h_p( \alpha \nu_1+(1-\alpha) \nu_2)\leq \alpha h_p(\nu_1)+(1-\alpha)h_p( \nu_2).$$
In order to show the other inequality, let us consider $ \varepsilon >0$  and we choose $ \eta_1$ so that
$$ h_p(\nu_1, \eta_1)> \begin{cases}
h_p(\nu_1)- \varepsilon \quad \hbox{if} \quad h_{p}(\nu_1)< \infty \\
\frac{1}{\varepsilon}  \quad \hbox{if} \quad h_{p}(\nu_1)=\infty
\end{cases}$$
and $\eta_2$ such that
$$ h_p(\nu_2, \eta_2)> \begin{cases}
h_p(\nu_2)- \varepsilon \quad \hbox{if} \quad h_{p}(\nu_2)< \infty \\
\frac{1}{\varepsilon}  \quad \hbox{if} \quad h_{p}(\nu_2)=\infty
\end{cases}$$
Putting $ \eta= \eta_1 \vee \eta_2$ in \eqref{entro1} we obtain
$$ h_p(\alpha \nu_1+(1- \alpha)\nu_2, \eta)> \begin{cases}
 \alpha h_p(\nu_1)+(1-\alpha)h_p( \nu_2)- \varepsilon \quad \hbox{if} \quad  h_{p}(\nu_1), h_{p}(\nu_2)< \infty\\
 \frac{1}{\varepsilon} \quad \hbox{if either} \quad  h_{p}(\nu_1)= \infty \quad \hbox{or} \quad h_{p}(\nu_2)= \infty.
\end{cases}$$
Hence $$h_p( \alpha \nu_1+(1-\alpha) \nu_2)\geq \alpha h_p(\nu_1)+(1-\alpha)h_p( \nu_2).$$
\end{proof}

\begin{theorem}
\label{ent2}
The quaternionic equilibrium measure $\mu$ is the unique invariant measure of measurable entropy equals to $\log d$.
\end{theorem}
\begin{proof}
From Lemma \ref{aff} we know that the measurable entropy is affine with respect to the measure. Since the quaternionic equilibrium measure $\mu$ was defined as $\frac{1}{4 \pi} \int_{I \in \mathbb{S}} \mu_I \,dI$ and since an affine function always commutes with the integral, we have that
$$h_p(\mu)=h_p \Big( \frac{1}{4 \pi} \int_{I \in \mathbb{S}} \mu_{I} \, dI \Big)= \frac{1}{4 \pi} \int_{I \in \mathbb{S}} h_{p}(\mu_{I}) \, dI= \frac{1}{4 \pi} \int_{I \in \mathbb{S}} \log d \, dI= \log d.$$

For the unicity it is sufficient to observe that if $\nu$ is another invariant measure of maximal entropy then $\nu_{|_{\mathbb{C}_{I}}}$ has to coincide with $\mu_I$, the complex equilibrium measure of $\mathbb{C}_I$, and hence $\nu \equiv \mu$.
\end{proof}
\begin{theorem}
\label{ent1}
The topological entropy of the quaternionic polynomial $p$ of degree $d$ with real coefficients is $ \log d$.
\end{theorem}
\begin{proof}
From Theorem \ref{Poen}, for $n=1$, we have $N(K_{I},1, \varepsilon)= d$, with $K_I \subseteq \mathbb{C}_I$, where $ K_I$ is the complex filled Julia set defined in Definition \ref{rr2}. By iterating we get $N(K_{I},n, \varepsilon)= d^{n}$.
Let $\varepsilon$ be the distances between the slices, then the number of planes that we need to affect the sphere is $ \frac{C}{\varepsilon}$, where $C$ is a positive constant, which depends on the measure of the sphere. From this argument it follows that
$$ N(K_{I} \times \mathbb{S}, n, \varepsilon)= d^{n} \frac{C}{ \varepsilon}.$$
Finally from Definition \ref{entropy} we have that
\begin{eqnarray*}
 h_{t}(p,K) =&& \sup_{\varepsilon>0} \overline{\lim_{n \to + \infty}} \frac{1}{n} \log \bigl(d^{n} \frac{C}{\varepsilon} \bigl)\\
=&& \sup_{\varepsilon>0} \overline{\lim_{n \to + \infty}} \frac{1}{n} \log (d^{n})+ \sup_{\varepsilon>0} \overline{\lim_{n \to + \infty}} \frac{1}{n} \log( \frac{C}{\varepsilon})= \log d.
\end{eqnarray*}
\end{proof}
From Theorems \ref{ent1} and \ref{ent2} it follows that
\begin{Cor}
The entropy of the equilibrium measure is equal to the topological entropy.
\end{Cor}

\begin{nb}
In Definition \ref{entro} if we consider as probability measure $ \nu$ the equilibrium measure $ \mu$ of a slice preserving polynomial $p$ and if we consider $C \in \xi$ the sets of the form $ \mathbb{S} \times A_I$ we have
$$ H(\mu, \xi)=- \sum_{C \in \xi} \mu(A_I) \log(\mu(A_I)),$$
and it is independent of I.
\end{nb}

\begin{prop}(Quaternionic variational principle)
If $K$ is a compact axially symmetric set containing the support of an invariant measure $\nu$, for $p$ a slice preserving polynomial we have
$$ h_p(\nu)\leq h_t(p,K).$$
\end{prop}
\begin{proof}
Let $ \eta= \{\mathbb{S} \times  C_1 ,..., \mathbb{S}  \times C_k  \}$ be a measurable partition of $K$, with the all $C_i$ in the same slice. Given $ \delta >0$, for each $i=1,...,k$, let $D_i$  be all in the same slice and be a compact subset of $ C_i$ such that $ \mathbb{S}  \times D_i \subset \mathbb{S} \times C_i  $ and $\nu((\mathbb{S} \times C_i )   \setminus (\mathbb{S}\times D_i) )< \delta.$ Now let $ \beta= \{\mathbb{S} \times D_0,...,  \mathbb{S} \times D_k\}$, where $D_0= K_I \setminus \bigcup_{i=1}^k D_i$. As in the complex case \cite[Thm. 1.1.5.37]{DS} \cite[Thm 4.7]{B1}
\begin{equation}
\label{E2}
h_p(\nu, \eta) <  h_p(\nu, \beta)+1.
\end{equation}
Now we construct a sequence of open covers
$$ \mathcal{U}= \{ \mathbb{S} \times (D_0 \cup D_1) ,...,  \mathbb{S} \times (D_0 \cup D_k) \},$$
as well as the covers
$$ \mathcal{U}_n =\bigl \{ \bigcap_{i=0}^{n-1} p^{-i} U_i: \, U_0,..., U_{n-1} \in \mathcal{U} \},$$
for each $n \in \mathbb{N}$. We call $L(n)$ the minimal number of sets in $ \mathcal{U}_n$ needed to cover $K$. We want to put in relation $L(n)$ and $ \beta_n= \bigvee_{i=0}^{n-1}p^{-i} \beta $. By the distributivity property of the cartesian product:
$$ p^{-i}[\mathbb{S} \times (D_0 \cup D_j)]=p^{-i}[(\mathbb{S} \times D_0 ) \cup (\mathbb{S} \times D_j)]=p^{-i}(\mathbb{S} \times D_0 ) \cup p^{-i}(\mathbb{S} \times D_j )$$
for $j=1,...,k$ and $i \in \mathbb{N}$. So any element of $ \mathcal{U}_n$ is the union of at most $2^n$ elements of $\beta_n$.  Moreover, since the elements in $\beta_n$ are disjoint, each element of $\beta_n$ is contained in any given subcover of $ \mathcal{U}_n$ . Therefore,
\begin{equation}
\# \beta_n \leq 2^n L(n)
\end{equation}
Furthermore, by the concavity of $-x \log x$, we obtain
\begin{equation}
\label{E3}
H(\nu, \beta_n) \leq \log \# \beta_n \leq n \log 2+ \log L(n).
\end{equation}
Moreover, as in the complex case,
\begin{equation}
\label{E4}
L(n) \leq N(K,n, \varepsilon).
\end{equation}
Thus, from \eqref{E3} and \eqref{E4} we have
\begin{equation}
\label{HP}
h_p(\nu, \beta)= \lim_{n \to \infty} \frac{1}{n} H(\nu, \beta_n) \leq \log 2+ \limsup_{n \to \infty} \frac{1}{n} \log N(K,n, \varepsilon).
\end{equation}
Substituting \eqref{HP} in \eqref{E2} and letting $ \varepsilon \to 0$ we get
\begin{equation}
h_{p}(\nu, \eta) \leq h_t(p,K)+ \log 2 +1
\end{equation}
hence
$$ h_{p}(\nu) \leq h_t(p,K)+ \log 2 +1.$$
Later we apply this inequality to $p^n$  and by \cite[Prop. 4.3, pag. 115]{B} and \cite[Prop. 3.1.7(3)]{KH}
we have
$$h_p(\nu)= \frac{1}{n} h_{p^n}(\nu) \leq \frac{1}{n}(h_t(p^n,K)+ \log 2 +1)=\frac{1}{n}(n h_t(p,K)+ \log 2 +1)= h_t(p,K)+ \frac{\log 2 +1}{n}$$
for each $n \in \mathbb{N}$. Letting $n \to \infty$, we conclude that $h_p(\nu) \leq h_t(p,K)$.
\end{proof}
By Theorem \ref{ent1}, we have the following result:
\begin{Cor}
For every invariant probability measure $ \nu$ and $p$ a slice preserving polynomial we have
$$ h_{p}(\nu) \leq  \log d.$$
\end{Cor}
This means that the quaternionic equilibrium measure, with respect to a slice preserving polynomial, is the measure of maximal entropy $ \log d$.

According to a theorem of Oseledec, \cite{O}:
\begin{theorem}
\label{split}
Let $\mathbb{H}$ be endowed with the quaternionic equilibrium measure $\mu$ with respect to a slice preserving polynomial $p$. Then, for almost all $q=x+yI \in \mathbb{H} \cong \mathbb{R}^{4}$, there exists a decomposition of $T_q \mathbb{R}^4= \mathbb{R}^4$ in real subspaces,
\begin{equation}
\label{CH1}
\mathbb{R}^4= \mathbb{R}^2_{(x,y)} \oplus T_{I}(\mathbb{S})
\end{equation}
and there exists a real number $ \lambda_1>0$ such that, for almost all $q=x+yI \in \mathbb{H}$
$$ \lim_{n \to + \infty} \frac{1}{n} \log| dp^{n}_q v|= \lambda_1$$
if $v \in T_{(x,y)}(\mathbb{R}^2)$. Let us define the projection
\begin{eqnarray}
\label{rr3}
&&\Pi: \mathbb{H}\setminus{\mathbb{R}} \to \mathbb{S}\\
\nonumber
&&q= \alpha+ I_q \beta \mapsto I_q
\end{eqnarray}
with $ \beta >0$. The properties of this projection are that $ \Pi\left(p(q)\right)= \Pi(q)=I_q$ and $ \Pi( \mathbb{S})= \mathbb{S}$. Therefore
$$ \lim_{n \to + \infty} \frac{1}{n} \log| d(\Pi \circ p^{n}_q) v|= 0$$
if $v \in T_{I}(\mathbb{S})$, where $d(\Pi \circ p^n_q)$ is the differential of $\Pi \circ p^n$ in $q$ with respect to the splitting \eqref{CH1}.
\end{theorem}
\begin{Cor}
The \emph{Lyapounov exponents} of $p$ with respect to $\mu$ are $\lambda_1$ and $0$.
\end{Cor}
\begin{nb}
Since the quaternionic measure $ \mu$ is mixing and ergodic (since it is so in all slices), the Lyapounov exponents of $p$ are constants.
\end{nb}
\begin{proof}(of Theorem \ref{split})
On each $ \mathbb{C}_I$, the behaviour of $p$ is the same of the corresponding 1-complex-variable polynomial $p_I$ of degree $d$.
\\ In the direction tangent to $ \mathbb{S}$, $p$ does not change, see \eqref{rr3}. This means that the polynomial is the identity, making the differential and evaluating the modulus we get $| d(\Pi \circ p^{n}_q) v|=1$ and finally since $\log | d(\Pi \circ p^{n}_q )v|=0$ we obtain that in the directions of $T_{I}(\mathbb{S})$ the Lyapounov exponent is $0$.
\end{proof}
\begin{nb}
Since $\lambda_1$ is positive and by theorem of Lyubich \cite{LY} it follows: $$\max \{ \lambda_1, 0 \}= \lambda_1 \geq \frac{1}{2} \log d>0.$$
\end{nb}
\subsection{Topological and measurable entropy: one slice preserving case}
Now, we evaluate the topological entropy and the measurable entropy with respect to the equilibrium measure in the one slice preserving case (see Theorem \ref{Brolin2}). In the following theorems $\mathcal{P}_I$ is a quaternionic polynomial with all coefficients in $ \mathbb{C}_I$ but not all real and $g_n$ is defined as in \eqref{gn}.
\begin{theorem}
The topological entropy of the quaternionic polynomial $\mathcal{P}_I$ is $ \log d$.
\end{theorem}
\begin{proof}
Using the same arguments of Theorem \ref{ent1} and since $g_n$ is slice preserving with degree $2d^n$ we have that
$$ N(K_{I} \times \mathbb{S}, n, \varepsilon)= 2 d^{n} \frac{C}{ \varepsilon}.$$
Finally from Definition \ref{entropy} we have that
\begin{equation}
\label{rr4}
h_{t}(g_n,K) =\sup_{\varepsilon>0} \overline{\lim_{n \to + \infty}} \frac{1}{n} \log \bigl(2 d^{n} \frac{C}{\varepsilon} \bigl)=\log d.
\end{equation}
\end{proof}

\begin{theorem}
The equilibrium measure of $\mathcal{P}_I$ has measurable entropy equals to $ \log d$.
\end{theorem}
\begin{proof}
By Theorem \ref{Brolin2} we know that for a polynomial with all coefficients in $ \mathbb{C}_I$ but not all real the equilibrium measure is of the following form  $ \frac{1}{8 \pi} \int_{I \in \mathbb{S}} \mu_{\mathfrak{P}(I)} dI + \frac{1}{8 \pi} \int_{I \in \mathbb{S}} \mu_{\mathfrak{P}^c (I)} dI$. Since by Lemma \ref{aff} the measurable entropy is affine we obtain
\begin{eqnarray*}
h_{g_n} \biggl( \frac{1}{8 \pi} \int_{I \in \mathbb{S}} \mu_{\mathfrak{P}(I)}dI + \frac{1}{8 \pi} \int_{I \in \mathbb{S}} \mu_{\mathfrak{P}^c (I)} dI  \biggl) &=& \frac{1}{4 \pi} h_{g_n} \biggl( \frac{1}{2} \int_{I \in \mathbb{S}} \mu_{\mathfrak{P}(I)} dI +\frac{1}{2} \int_{I \in \mathbb{S}}  \mu_{\mathfrak{P}^c (I)} dI  \biggl)\\
&=& \frac{1}{4 \pi}  \biggl( \frac{1}{2} \int_{I \in \mathbb{S}} h_{g_n}(\mu_{\mathfrak{P} (I)}) dI +\frac{1}{2} \int_{I \in \mathbb{S}} h_{g_n}(\mu_{\mathfrak{P}^c (I)})dI  \biggl)\\
&=& \frac{1}{8 \pi}  ( 4 \pi \log d + 4 \pi \log d)= \log d,
\end{eqnarray*}
where the last equality comes from \eqref{rr4}.
\end{proof}
As in the slice preserving case  we have the following result.
\begin{Cor}
The entropy of the equilibrium measure of $\mathcal{P}_I$ is equal to the topological entropy.
\end{Cor}

\hspace{2mm}

\noindent
Cinzia Bisi,
Dipartimento di Matematica e Informatica \\Universit\`a di Ferrara\\
Via Ma\-chia\-vel\-li n.~30\\
I-44121 Ferrara\\
Italy

\noindent
\emph{email address}: bsicnz@unife.it\\
\emph{ORCID iD}: 0000-0002-4973-1053

\vspace*{5mm}
\noindent
Antonino De Martino,
Dipartimento di Matematica \\ Politecnico di Milano\\
Via Bonardi n.~9\\
20133 Milano\\
Italy

\noindent
\emph{email address}: antonino.demartino@polimi.it\\
\emph{ORCID iD}: 0000-0002-8939-4389

\end{document}